\documentclass[a4j,10pt,showkeys]{article}

\usepackage{amsmath}    
\usepackage{graphicx}   
\usepackage{geometry}
\usepackage{verbatim}   
\usepackage{color}      
\usepackage{hyperref}   

\usepackage{amssymb}
\usepackage{amsthm}
\usepackage{mathrsfs}

\newsavebox{\smlmat}

\newtheorem{Def}{Definition}[section]

\newtheorem{Prop}[Def]{Proposition}
\newtheorem{Lem}[Def]{Lemma}
\newtheorem{Thm}[Def]{Theorem}
\newtheorem{Cor}[Def]{Corollary}

\theoremstyle{definition}
\newtheorem{Rem}[Def]{Remark}

\newcommand{\p}{\mathbb{P}}
\newcommand{\e}{\mathbb{E}}
\newcommand{\real}{\mathbb{R}}
\newcommand{\n}{\mathbb{N}}

\newcommand{\1}{{\bf 1}}

\newcommand{\rd}{\,\mathrm{d}}

\allowdisplaybreaks

\begin{document}
\title{
Semi-implicit Euler--Maruyama scheme for polynomial diffusions on the unit ball
}
\author{
Takuya Nakagawa\footnote{
Department of Mathematical Sciences,
Ritsumeikan University,
1--1--1,
Nojihigashi, Kusatsu,
Shiga, 525--8577, Japan,
email~:~\texttt{takuya.nakagawa73@gmail.com}
},~
Dai Taguchi\footnote{
Research Institute for Interdisciplinary Science, department of mathematics,
Okayama University,
3--1--1
Tsushima-naka,
Kita-ku
Okayama
700--8530,
Japan,
email~:~\texttt{dai.taguchi.dai@gmail.com}
}
~and~
Tomooki Yuasa\footnote{
Faculty of Economics and Business Administration, Graduate School of Management,
Tokyo Metropolitan University,
Marunouchi Eiraku Building 18F,
1--4--1,
Marunouchi,
Chiyoda-ku,
Tokyo,
100--0005,
Japan,
email~:~\texttt{to-yuasa@tmu.ac.jp}
}
}
\maketitle
\begin{abstract}
In this article, we consider numerical schemes for polynomial diffusions on the unit ball, which are solutions of stochastic differential equations with a diffusion coefficient of the form $\sqrt{1-|x|^{2}}$. We introduce a semi-implicit Euler--Maruyama scheme with the projection onto the unit ball and provide the $L^{2}$-rate of convergence. The main idea to consider the numerical scheme is the transformation argument introduced by Swart \cite{Sw02} for proving the pathwise uniqueness for some stochastic differential equation with a non-Lipschitz diffusion coefficient.
\end{abstract}


\section{Introduction}
\label{sec:1}
In this article, we study numerical schemes for the stochastic differential equation (SDE)
\begin{align}\label{SDE_ball_0}
\rd X(t)
=
-\kappa X(t) \rd t
+
\nu\sqrt{1-|X(t)|^{2}} \rd W(t),~
X(0)=x(0) \in \mathscr{B}^{d}
\end{align}
(or more general SDE \eqref{eq:1}) on the unit ball $\mathscr{B}^{d}$, where $W=(W_{1},\ldots,W_{d})^{\top}$ is a standard $d$-dimensional Brownian motion, and $\kappa$ and $\nu$ are non-negative constants.
The SDE \eqref{SDE_ball_0} is one of multi-dimensional extensions of Jacobi processes or Wright--Fisher processes, and is a special case of polynomial diffusions on the unit ball.
Polynomial diffusions have been applied to various financial models (e.g. \cite{AcFi20,CuKeReTe12,DuFiSc03,FiLa16,GoJa06,LaSo07}).
When it comes to actually using their financial models in practice, it is often necessary to generate random numbers of the polynomial diffusions.
In general, however, it is difficult to generate random numbers directly since explicit forms of solutions of SDEs are not always found.
Therefore, it is important from a practical point of view to consider numerical schemes for solutions of SDEs whose random numbers one can generate.

Before introducing our numerical scheme for the polynomial diffusion \eqref{SDE_ball_0} (or \eqref{eq:1}), we recall the definition and some properties of polynomial diffusions.
For continuous maps $a=(a_{i,j})_{i,j=1}^{d}:\real^{d} \to \mathbb{S}^{d}$ and $b=(b_{1},\ldots,b_{d})^{\top}:\real^{d} \to \real^{d}$ with $a_{i,j} \in \mathrm{Pol}_{2}$ and $b_{i} \in \mathrm{Pol}_{1}$, $i,j=1,\ldots,d$, a time-homogeneous Markov process $X=(X(t))_{t \geq 0}$ is called a \emph{polynomial diffusion} if it is a solution of the SDE $\rd X(t)=b(X(t)) \rd t + \sigma(X(t)) \rd B(t)$, where $\sigma:\real^{d} \to \real^{d \times n}$ is a continuous map with $\sigma \sigma^{\top}=a$ and $B$ is a standard $n$-dimensional Brownian motion.
Here $\mathbb{S}^{d}$ denotes the set of $d \times d$ real symmetric matrices and $\mathrm{Pol}_{k}$ denotes the vector space of polynomial functions on $\real^{d}$ of total degree less than or equal to $k$.
Then the infinitesimal generator $\mathscr{L}$ of $X$, defined by $\mathscr{L}f=\langle b,\nabla f \rangle+\mathrm{Tr}(a\nabla^{2}f)/2$, preserves that $\mathscr{L} \mathrm{Pol}_{k} \subset \mathrm{Pol}_{k}$ for all $k \in \n$.
Polynomial diffusions have been widely studied from both theory and practical points of view (e.g. one-dimensional polynomial diffusions such as Ornstein--Uhlenbeck processes, geometric Brownian motions, Cox--Ingersoll--Ross (CIR) processes (squared Bessel processes), Wright--Fisher processes and Jacobi processes have been applied in mathematical finance and biology \cite{KaTa81,LaSo07}, and multi-dimensional polynomial diffusions such as affine processes and Jacobi processes have been applied in \cite{AcFi20,CuKeReTe12,DuFiSc03,FiLa16,GoJa06,Sw02}. 

Many of polynomial diffusions used in applications have boundary conditions.
For instance, the state spaces of CIR processes, Wright--Fisher processes and Jacobi processes are $[0,\infty)$, $[0,1]$ and $[-1,1]$, respectively.
Now, we consider multi-dimensional polynomial diffusions whose state spaces are either the unit ball $\mathscr{B}^{d}:=\{x \in \real^{d}\,;\,|x| \leq 1\}$ or the unit sphere $\mathscr{S}^{d-1}:=\{x \in \real^{d}\,;\,|x| \leq 1\}$.
This covers all nondegenerate compact quadric sets up to an affine change of coordinates (see Section 2 in \cite{LaPu17}).
Filipovi\'c and Larsson \cite{FiLa16}, and Larsson and Pulido \cite{LaPu17} characterized weak solutions of the polynomial diffusion $\rd X(t)=b(X(t)) \rd t + \sigma(X(t)) \rd B(t)$ on the unit ball $\mathscr{B}^{d}$ as follows.
It is shown in Proposition 6.1 of \cite{FiLa16} or Theorem 2.1 of \cite{LaPu17} that the SDE admits a $\mathscr{B}^{d}$-valued weak solution for any initial condition in $\mathscr{B}^{d}$ if and only if the coefficients $b$ and $\sigma$ are of the form
\begin{align*}
b(x)
=
b+\beta x
\quad\text{and}\quad
a(x)
=
(1-|x|^{2}) \alpha
+
c(x),~x \in \mathscr{B}^{d}
\end{align*}
for some $b \in \real^{d}$, $\beta \in \real^{d \times d}$, $\alpha \in \mathbb{S}^{d}_{+}$ and $c \in \mathscr{C}_{+}$ such that
\begin{align*}
\langle b,x \rangle
+
\langle \beta x,x \rangle
+
\frac{1}{2}
\mathrm{Tr}(c(x))
\leq 0,~x \in \mathscr{S}^{d-1}.
\end{align*}
Here $\mathbb{S}^{d}_{+}$ denotes the set of $d \times d$ real symmetric positive semidefinite matrices and $\mathscr{C}_{+}:=\{c:\real^{d} \to \mathbb{S}^{d}_{+}\,;\,c_{i,j} \in \mathrm{Pol}_{2}~\text{is homogeneous of degree}~2~\text{for all}~i,j,~\text{and}~c(x)x=0\}$.
Moreover, under the above conditions, Larsson and Pulido \cite{LaPu17} provided an SDE representation by using a $\mathrm{Skew}(d)$-valued correlated Brownian motion with drift as follows.
Here $\mathrm{Skew}(d)$ denotes the set of $d\times d$ real skew-symmetric matrices.
Let  $m:=\begin{pmatrix} d \\ 2\end{pmatrix}=\dim \mathrm{Skew}(d)$ and $A_{1},\ldots,A_{m} \in \mathrm{Skew}(d)$.
Then they proved that the above map $c$ is of the form $c(x)=\sum_{p=1}^{m}A_{p} x x^{\top} A_{p}^{\top} \in \mathrm{Skew}(d)$ if and only if the law of the above polynomial diffusion is the same to the one of the solution of the SDE:
\begin{align}\label{eq:0}
\mathrm{d} X(t)
&=
(b+\widehat{\beta}X(t))\rd t
+
\sqrt{1-|X(t)|^{2}} \alpha^{1/2} \rd W(t)
+
A_{0}X(t)\rd t
+
\sum_{p=1}^{m}A_{p}X(t) \circ \rd \widehat{W}_{p}(t)
\end{align}
on the unit ball $\mathscr{B}(\real^{d})$, where $B=(W_{1},\ldots,W_{d},\widehat{W}_{1},\ldots,\widehat{W}_{m})^{\top}$ is a standard $(d+m)$-dimensional Brownian motion, $A_{0}=2^{-1}(\beta-\beta^{\top}) \in \mathrm{Skew}(d)$ and $\widehat{\beta}=2^{-1}(\beta+\beta^{\top})+2^{-1}\sum_{p=1}^{m}A_{p}^{\top}A_{p} \in \mathbb{S}^{d}$
(see Lemma 3.4 and Theorem 4.2 in \cite{LaPu17}).
Here the stochastic process $A_{0}t+\sum_{p=1}^{m}A_{p}\widehat{W}_{p}(t)$ is called a $\mathrm{Skew}(d)$-valued correlated Brownian motion with drift and the notation $\circ \rd \widehat{W}_{p}(t)$ means the Stratonovich integral.
In other words, rewriting it in the It\^o integral yields $A_{p}X(t) \circ \rd \widehat{W}_{p}(t)=A_{p}X(t)\rd \widehat{W}_{p}(t)+(1/2)A_{p}^{2}X(t) \rd t$.
Moreover, by the skew symmetry of $A_{1}, \ldots, A_{m}$, the SDE \eqref{eq:0} coincides with the SDE $\mathrm{d} X(t)=(b+\beta X(t))\rd t+\widehat{\sigma}(X(t))\rd B(t)$, where $\widehat{\sigma}(x):=(\sqrt{1-|x|^{2}} \alpha^{1/2},A_{1}x,\ldots, A_{m}x ) \in \real^{d \times (d+m)}$ which satisfies $\widehat{\sigma} \widehat{\sigma}^{\top}=a$.
Thus the law of the solution of the SDE \eqref{eq:0} is the same to the one of the polynomial diffusion with $b(x)=b+\beta x$ and $a(x)=(1-|x|^{2}) \alpha +c(x)$ (see Corollary 3.2 in \cite{FiLa16} or Lemma 2.4 in \cite{LaPu17}).

For one-dimensional SDEs (not only polynomial diffusions), Yamada and Watanabe \cite{YaWa} proved that if the drift coefficient is Lipschitz continuous and the diffusion coefficient is $1/2$-H\"older continuous, then the pathwise uniqueness holds.
However, in the multi-dimensional case, it is difficult to apply their method to derive the pathwise uniqueness, and thus this is one of significant issues in the field of stochastic calculus.
On one hand, as a special case of polynomial diffusions, Swart \cite{Sw02} showed that the pathwise uniqueness holds for the SDE \eqref{SDE_ball_0} on the unit ball $\mathscr{B}^{d}$ with $\kappa \geq 1$ and $\nu=\sqrt{2}$ by using the transformation $\mathscr{B}^{d} \ni x \mapsto (\sqrt{1-|x|^{2}},x_{1},\ldots,x_{d})^{\top} \in [0,1] \times \mathscr{B}^{d}$.
More preciously, for the solution $X$ of the SDE \eqref{SDE_ball_0}, we define the $(d+1)$-dimensional stochastic process $Y=(Y_{0}, Y_{1}, \ldots,Y_{d})^{\top}:=(\sqrt{1-|X|^{2}}, X_{1},\ldots,X_{d})^{\top}$ which takes values in the upper-half ball surface $[0,1] \times \mathscr{B}^{d}$, and then the pathwise uniqueness holds for $Y$ (see Theorem 3 in \cite{Sw02}).
On the other hand, DeBlassie \cite{De04} extended Swart's result to more general SDEs with parameters $\kappa$ and $\nu=\sqrt{2}$ replaced by Lipschitz continuous functions satisfying certain conditions.
The idea is to use the stochastic process $(1-|X|^{2})^{p}$ with suitable $p \in (1/2,1)$ instead of $\sqrt{1-|X|^{2}}$.
Furthermore, by using DeBlassie's method, Larsson and Pulido \cite{LaPu17} generalized Swart's result to the following SDE which is a special case of the SDE \eqref{eq:0}:
\begin{align}
\label{eq:1}
\rd X(t)
=
-\kappa X(t)\rd t
+
\nu
\sqrt{1-|X(t)|^{2}} \rd W(t)
+
A_{0}X(t)\rd t
+
\sum_{p=1}^{m}A_{p}X(t) \circ \rd \widehat{W}_{p}(t)
\end{align}
on the unit ball $\mathscr{B}^{d}$, where $B=(W_{1},\ldots,W_{d},\widehat{W}_{1},\ldots,\widehat{W}_{m})^{\top}$ is a standard $(d+m)$-dimensional Brownian motion, and $\kappa$ and $\nu$ are non-negative constants.
To be specific, they showed that if $\kappa/\nu^{2}>\sqrt{2}-1$, then the pathwise uniqueness holds for the SDE \eqref{eq:1} (see Theorem 4.6 in \cite{LaPu17}).
In this article, we will provide a numerical scheme on the unite ball $\mathscr{B}^{d}$ for the polynomial diffusion \eqref{eq:1}.

As mentioned above, a solution $X=(X(t))_{t \in [0,T]}$ of the SDE $\rd X(t)=b(X(t)) \rd t + \sigma(X(t)) \rd B(t)$ (not only polynomial diffusions) does not always have explicit forms in general, and thus one often approximates it by using the Euler--Maruyama scheme $X^{(n)}_{\mathrm{EM}}=(X^{(n)}_{\mathrm{EM}}(t_{k}))_{k=0,1,\ldots,n}$ defined by
\begin{align*}
X^{(n)}_{\mathrm{EM}}(t_{k+1})
=
X^{(n)}_{\mathrm{EM}}(t_{k})
+
b(X^{(n)}_{\mathrm{EM}}(t_{k}))
\Delta t
+
\sigma(X^{(n)}_{\mathrm{EM}}(t_{k}))
\Delta_{k} B
\end{align*}
for $k=0,1,\ldots,n-1$ with the initial condition $X^{(n)}(0)=X(0)$.
Here $\Delta t:=T/n$, $t_{k}:=k\Delta t$ and $\Delta_{k} B :=B(t_{k+1})-B(t_{k})$.
On one hand, it is well-known that under the Lipschitz condition on the coefficients $b$ and $\sigma$, the $L^p$-rate of convergence for the Euler--Maruyama scheme $X^{(n)}_{\mathrm{EM}}$ is $1/2$, that is, for any $p \geq 1$, it holds that $\e[\max_{k=0,1,\ldots,n}|X(t_{k})-X^{(n)}_{\mathrm{EM}}(t_{k})|^{p}]^{1/p} \leq C_{p} n^{-1/2}$ (see \cite{KP}).
Moreover, for SDEs with reflecting boundary conditions, the projection scheme and the penalization scheme which are based on the Euler--Maruyama scheme have been studied (see 
\cite{BeBoDi,Le95,Pe97,Pe20,Sl94,Sl01,Sl13}).
On the other hand, Kaneko and Nakao \cite{KaNa88} showed that by using Skorokhod's arguments \cite{Sk65}, if the pathwise uniqueness holds for SDEs with continuous and linear growth coefficients, then the Euler--Maruyama scheme converges to a solution of the corresponding SDE in the $L^{2}$ sense.
Moreover, for one-dimensional SDEs with an $\alpha$-H\"older continuous diffusion coefficient with $\alpha \in [1/2,1]$, Yan \cite{Ya02}, and Gy\"ongy and R\'asonyi \cite{GyRa11} provided the $L^{1}$-rate of convergence for the Euler--Maruyama scheme by using It\^o--Tanaka's formula, and Yamada and Watanabe's approximation arguments, respectively.
For SDEs with boundary conditions (e.g. CIR processes, Wright--Fisher processes and Jacobi processes, and especially \eqref{eq:1}), there is a problem that the Euler--Maruyama scheme does not always take values in the state space of the corresponding SDE.
For solving this problem, the Lamperti transformation and the implicit Euler--Maruyama scheme play crucial roles in the one-dimensional setting.
To explain this, for example, we consider the CIR process $y=(y(t))_{t \in [0,T]}$ which is the solution of the SDE $\rd y(t)=(a-by(t))\rd t+\sigma \sqrt{y(t)} \rd W(t)$ with the initial condition $y(0) >0$ and parameters $a,b,\sigma \in \real$.
If $2a \geq \sigma^{2}$, then it holds that $\p(y(t) \in (0,\infty),~\forall t \in [0,T])$=1.
By using the Lamperti transformation, that is, applying It\^o's formula to $x(t):=\sqrt{y(t)}$, the stochastic process $x=(x(t))_{t \in [0,T]}$ satisfies the SDE $\rd x(t)=(a/2-\sigma^{2}/8)x(t)^{-1}-(b/2)x(t) \rd t+(\sigma/2) \rd W(t)$.
Then the implicit Euler--Maruyama scheme $x^{(n)}=(x^{(n)}(t_{k}))_{k=0,1,\ldots,n}$ for $x$ is defined by the unique positive solution of the quadratic equation
\begin{align*}
x^{(n)}(t_{k+1})
=
x^{(n)}(t_{k})
+
\left\{
\left(\frac{a}{2}-\frac{\sigma^{2}}{8}\right)
\frac{1}{x^{(n)}(t_{k+1})}
-
\frac{b}{2}x^{(n)}(t_{k+1})
\right\}
\Delta t
+
\frac{\sigma}{2}
\Delta_{k} W
\end{align*}
for $k=0,1,\ldots,n-1$ with the initial condition $x^{(n)}(0)=x(0)$.
Then it is shown that the inverse transform of $x^{(n)}$ converges to the solution $y$ (for more details, see \cite{Al13,DeNeSz,NeSz} and \cite{HMS,Hu96}).

In this article, we will provide a numerical scheme on the unit ball $\mathscr{B}^{d}$ for the solution of the multi-dimensional SDE \eqref{eq:1} with the initial condition $X(0)=x(0) \in \mathscr{B}^{d} \setminus \mathscr{S}^{d-1}$.
As mentioned above, the solution takes values in the unit ball $\mathscr{B}^{d}$, but the Euler--Maruyama scheme does not always.
Moreover, unfortunately, in the multi-dimensional setting, it is difficult to use the Lamperti transformation, unlike in the one-dimensional setting.
Therefore, as an alternative method, we first consider the transformed stochastic process $Y=(\sqrt{1-|X|^{2}},X_{1},\ldots,X_{d})^{\top}$ introduced by Swart \cite{Sw02}, and then we approximate it by a semi-implicit Euler--Maruyama scheme as follows.
It can be shown that by using It\^o's formula and the skew-symmetry of $A_{0},\ldots,A_{m}$, $Y=(Y_{0},X_{1},\ldots,X_{d})^{\top}$ satisfies the following $(d+1)$-dimensional SDE:
\begin{align}
\label{eq:4}
\begin{split}
\rd Y_{0}(t)
&=
\Big\{
	\frac{\kappa-\frac{\nu^{2}}{2}}{Y_{0}(t)}
	-
	\Big(
		\kappa-\frac{\nu^{2}}{2}
		+
		\frac{d\nu^{2}}{2}
	\Big)
	Y_{0}(t)
\Big\} \rd t
-
\nu X(t)^{\top} \rd W(t),
\\
\rd X(t)
&=
-\kappa X(t)\rd t
+
\nu
Y_{0}(t) \rd W(t)
+
A_{0}X(t)\rd t
+
\sum_{p=1}^{m}A_{p}X(t)  \rd \widehat{W}_{p}(t)
+
\frac{1}{2}
\sum_{p=1}^{m}A_{p}^{2}X(t)  \rd t
\end{split}
\end{align}
(see Section \ref{sec:2}).
This implies that the stochastic process $Y_{0}$ is a ``Bessel like'' process.
Inspired by previous studies \cite{Al13,DeNeSz,NeSz}, we introduce a semi-implicit Euler--Maruyama scheme for the system of SDE \eqref{eq:4} as follows.
Let $\Delta t:=T/n$, $t_{k}:=k\Delta t$, $k=0,1,\ldots,n$, $\Delta_{k} W:=W(t_{k+1})-W(t_{k})$ and $\Delta_{k} \widehat{W}:=\widehat{W}(t_{k+1})-\widehat{W}(t_{k})$ and assume  that $\kappa/\nu^{2} > 1/2$.
Define $Y^{(n)}(0):=Y(0)$ and $Y^{(n)}(t_{k+1})=(Y_{0}^{(n)}(t_{k+1}),Y_{1}^{(n)}(t_{k+1}),\ldots,Y_{d}^{(n)}(t_{k+1}))^{\top}:=(Y_{0}^{(n)}(t_{k+1}),X_{1}^{(n)}(t_{k+1}),\ldots,X_{d}^{(n)}(t_{k+1}))^{\top}$, $k=0,1,\ldots,n-1$, as the unique solution in $(0,\infty) \times \real^{d}$ (does not always take values in $(0,1) \times \mathscr{B}^{d}$) of the following equation:
\begin{align}\label{BEM_0}
\begin{split}
Y_{0}^{(n)}(t_{k+1})
&=
Y_{0}^{(n)}(t_{k})
+
\frac{\kappa-\frac{\nu^{2}}{2}}{Y_{0}^{(n)}(t_{k+1})}\Delta t
-
\Big(
	\kappa
	-
	\frac{\nu^{2}}{2}
	+
	\frac{d\nu^{2}}{2}
\Big)
Y_{0}^{(n)}(t_{k})
\Delta t
-
\nu X^{(n)}(t_{k})^{\top} \Delta_{k}W,
\\
X^{(n)}(t_{k+1})
&=
X^{(n)}(t_{k})
-
\kappa X^{(n)}(t_{k})\Delta t
+
\nu Y_{0}^{(n)}(t_{k})\Delta_{k}W
\\&\hspace{0.4cm}
+
A_{0}X^{(n)}(t_{k})\Delta t
+
\sum_{p=1}^{m}A_{p}X^{(n)}(t_{k})\Delta_{k}\widehat{W}_{p}
+
\frac{1}{2}\sum_{p=1}^{m}A_{p}^{2}X^{(n)}(t_{k})\Delta t.
\end{split}
\end{align}
In particular, $Y_{0}^{(n)}(t_{k+1})$ has the following explicit form:
\begin{align*}
\begin{split}
Y_{0}^{(n)}(t_{k+1})&
=
\frac{1}{2}
\Big\{
	b_{k}
	+
	\sqrt{
		b_{k}^{2}
		+
		4\Big(
			\kappa
			-
			\frac{\nu^{2}}{2}
		\Big)
		\Delta t
	}
\Big\}, \\
b_{k}
&:=
Y_{0}^{(n)}(t_{k})
-
\nu X^{(n)}(t_{k})^{\top}
\Delta_{k}W
-
\Big(
	\kappa
	-
	\frac{\nu^{2}}{2}
	+
	\frac{d\nu^{2}}{2}
\Big)
Y_{0}^{(n)}(t_{k})
\Delta t.
\end{split}
\end{align*}
We will provide the $L^{2}$-rate of convergence for $Y^{(n)}$ to the solution $Y$ of the system of SDE \eqref{eq:4} (see Theorem \ref{thm:2.3}).
This rate can theoretically be applied to the computational complexity of the multilevel Monte Carlo method, whose computational cost is much lower than that of the classical (single level) Monte Carlo method (see \cite{Gi08}).
To the best of our knowledge, the $L^{2}$-rate of convergence for numerical schemes of multi-dimensional SDEs with a H\"older continuous and degenerate diffusion coefficient have not been obtained yet.
Here note that $X^{(n)}$ may still take values outside of the unit ball $\mathscr{B}^{d}$.
We solve this problem by projecting it onto the unit ball $\mathscr{B}^{d}$ as follows.
Let $\Pi$ be the projection onto the unit ball $\mathscr{B}^{d}$ defined by $\Pi(x):=x{\bf 1}_{\mathscr{B}^{d}}(x)+(x/|x|){\bf 1}_{{\mathbb R}^{d} \setminus \mathscr{B}^{d}}(x)$, $x \in {\mathbb R}^{d}$.
Then we define the projection scheme $\overline{X}^{(n)}=(\overline{X}^{(n)}(t_{k}))_{k=0,1,\ldots,n}$ on the unit ball $\mathscr{B}^{d}$ for the SDE \eqref{eq:1} by
\begin{align}
\label{eq:8}
\overline{X}^{(n)}(t_{k}):=\Pi(X^{(n)}(t_{k})).
\end{align}
We will show that the $L^{2}$-rate of convergence for $\overline{X}^{(n)}$ to $X$ is induced while preserving the $L^{2}$-rate of convergence for $Y^{(n)}$ to $Y$ (see Corollary \ref{thm:2.1}).

This article is structured as follows.
In Section \ref{sec:2}, we provide the $L^{2}$-rates of convergence for the semi-implicit Euler--Maruyama scheme $Y^{(n)}$ defined in \eqref{BEM_0} and for the projection scheme $\overline{X}^{(n)}$ defined in \eqref{eq:8} (see Theorem \ref{thm:2.1} and Corollary \ref{thm:2.3}).
In Section \ref{sec:3}, we provide some numerical results about the projection scheme $\overline{X}^{(n)}$ for the polynomial diffusions \eqref{SDE_ball_0} and \eqref{eq:1}, and about the difference between two solutions of the system of SDE \eqref{eq:4} using the projection scheme $\overline{X}^{(n)}$.

\subsection*{Notations}\label{sec:1.3}
We give some basic notations and definitions used throughout this article.
Each element of $\real^{d}$ is understood as a column vector, that is, $x=(x_1,\ldots,x_d)^{\top}$ for $x \in \real^d$.
For a $d \times d$ real matrix $A=(A_{i,j})_{i,j=1,\ldots,d}$, the transpose of $A$ is denoted by $A^{\top}$, and the Frobenius norm of $A$ is denoted by $\|A\|:=(\sum_{i,j=1}^{d}A_{i,j}^{2})^{1/2}$.
We define $\mathbb{S}^{d}$ to be the set of $d \times d$ real symmetric matrices and $\mathrm{Skew}(d)$ to be the set of $d\times d$ real skew-symmetric matrices, that is, $A^{\top}=-A$ for $A \in \mathrm{Skew}(d)$.

The unit ball and the unit sphere are defined by $\mathscr{B}^{d}:=\{x \in \real^{d}\,;\,|x| \leq 1\}$ and $\mathscr{S}^{d-1}:=\{x \in \real^{d}\,;\,|x| = 1\}$, respectively.
Let $\Pi :\real^{d} \to \mathscr{B}^{d}$ be the projection defined by $\Pi(x):=x \1_{\mathscr{B}^{d}}(x)+(x/|x|) \1_{\real^{d} \setminus \mathscr{B}^{d}}(x)$.

Let $B=(W_{1},\ldots,W_{d},\widehat{W}_{1},\ldots,\widehat{W}_{m})^{\top}$ be a standard $(d+m)$-dimensional Brownian motion on a complete probability space $(\Omega,\mathscr{F},\p)$ with a filtration $(\mathscr{F}(t))_{t\geq 0}$ satisfying the usual conditions.
For $T>0$ and $n \in \n$, we denote $\Delta t:=T/n$, $t_{k}:=k\Delta t$, $k=0,1,\ldots,n$, $\Delta_{k} W:=W(t_{k+1})-W(t_{k})$ and $\Delta_{k} \widehat{W}:=\widehat{W}(t_{k+1})-\widehat{W}(t_{k})$.

\section{Main results}
\label{sec:2}

Let $d \geq 2$ and $T>0$ be fixed, and let $X=(X(t))_{t \in [0,T]}$ be a solution of the SDE \eqref{eq:1}.
In this article, we assume that $\kappa /\nu^{2} \geq \sqrt{2}-1$ and $A_{1},\ldots,A_{m} \in \mathrm{Skew}(d)$.
Then the pathwise uniquness holds for the SDE \eqref{eq:1} (see Theorem 4.6 in \cite{LaPu17}) and $\p(X(t) \in \mathscr{B}^{d},~\forall t \in [0,T])=1$ (see Theorem 2.1 in \cite{LaPu17}).
Furthermore, we assume that the initial condition $X(0)=x(0)$ takes a deterministic value in $\mathscr{B}^{d}\setminus \mathscr{S}^{d-1}$.

In this section, we provide the $L^{2}$-rates of convergence for the semi-implicit Euler--Maruyama scheme \eqref{BEM_0} and the projection scheme \eqref{eq:8}.
We first recall the transformation argument introduced by Swart \cite{Sw02}.
We define a transposed stochastic process $Y=(Y_{0},Y_{1},\ldots,Y_{d})^{\top}:=(\sqrt{1-|X|^{2}},X_{1},\ldots,X_{d})^{\top}$.
Then $Y=(Y_{0},X_{1},\ldots,X_{d})^{\top}$ satisfies the SDE \eqref{eq:4}.
Indeed, by using It\^o's formula, we obtain
\begin{align*}
\rd |X(t)|^{2}
&=2\nu \sqrt{1-|X(t)|^{2}} X(t)^{\top} \rd W(t)
+2\sum_{p=1}^{m} \langle A_{p}X(t),X(t) \rangle \rd \widehat{W}_{p}(t) \\ \notag
&\hspace{0.4cm}
-2\kappa |X(t)|^{2}\rd t+2 \langle A_{0}X(t),X(t) \rangle \rd t
+\sum_{p=1}^{m} \langle A_{p}^{2}X(t),X(t) \rangle\rd t \\ \notag
&\hspace{0.4cm}
+d\nu^{2}(1-|X(t)|^{2})\rd t
+\sum_{p=1}^{m}|A_{p}X(t)|^{2}\rd t.
\end{align*}
Then since for any $A \in \mathrm{Skew}(d)$,
\begin{align}
\label{eq:3}
\langle A x,x \rangle=0
~\text{ and }~
\langle A^{2}x, x \rangle
=-|Ax|^{2},~x \in \real^{d},
\end{align}
we have
\begin{align}
\label{eq:12}
\rd |X(t)|^{2}
=
2\nu \sqrt{1-|X(t)|^{2}} X(t)^{\top} \rd W(t)
+
\left\{d\nu^{2}-(d\nu^{2}+2\kappa)|X(t)|^{2}\right\}\rd t.
\end{align}
Hence by using It\^o's formula again, it holds that
\begin{align*}
\rd Y_{0}(t)
&=
\Big\{
	\frac{\kappa-\frac{\nu^{2}}{2}}{Y_{0}(t)}
	-
	\Big(
		\kappa-\frac{\nu^{2}}{2}
		+
		\frac{d\nu^{2}}{2}
	\Big)
	Y_{0}(t)
\Big\} \rd t
-
\nu X(t)^{\top} \rd W(t).
\end{align*}
Therefore, $Y=(Y_{0},X_{1},\ldots,X_{d})^{\top}$ is a solution of the $(d+1)$-dimensional SDE \eqref{eq:4}.

Recall that $Y^{(n)}$ is the semi-implicit Euler--Maruyama scheme defined in \eqref{BEM_0} for $Y$ and $\overline{X}^{(n)}$ is the projection scheme defined in \eqref{eq:8} for $X$.
We first provide the $L^{2}$-rate of convergence for $Y^{(n)}$.




\begin{Thm}
\label{thm:2.3}
Suppose $\kappa/\nu^{2}>3$.
Then there exists $C>0$ such that for any $n \in \n$,
\begin{align*}
\max_{k=0,1,\ldots,n}{\mathbb E}\left[\left|Y(t_{k})-Y^{(n)}(t_{k})\right|^{2}\right]^{1/2} \leq \frac{C}{n^{1/2}} \quad \text{ and } \quad {\mathbb E}\left[\max_{k=0,1,\ldots,n}\left|Y(t_{k})-Y^{(n)}(t_{k})\right|^{2}\right]^{1/2} \leq \frac{C}{n^{1/4}}.
\end{align*}
\end{Thm}

\begin{Rem}\label{Rem_0}
\begin{itemize}
\item[(i)]
In the case of one-dimensional SDEs with a constant diffusion coefficient, it is proven in \cite{Al13,DeNeSz,NeSz} that the $L^{2}$-sup rate of convergence for the implicit Euler--Maruyama scheme is $1/2$ or $1$.
However, in our case, the diffusion coefficient of the equation $Y_{0}$ in the system of SDE \eqref{eq:4} is not constant, and thus we need to estimate the supremum for each of the martingale terms $R^{M}$ and $S^{M}$ defined below, which does not appear in the above case.
This is the reason that the $L^{2}$-sup rate of convergence in Theorem \ref{thm:2.3} is $1/4$.

\item[(ii)]
The assumption $\kappa/\nu^{2}>3$ will be used in \eqref{eq:17.1} to apply Proposition \ref{lem:2.2} with $q=6<2\kappa/\nu^{2}$.
This assumption is exactly the same as in the case of the one-dimensional setting considered in Section 3.5 of \cite{NeSz}.
Indeed, let $X=(X(t))_{t \in [0,T]}$ be a solution of SDE \eqref{SDE_ball_0} with $d=1$.
Then $y(t):=X(t)^{2}$ is a solution of the Wright--Fisher diffusion
\begin{align*}
\rd y(t)=(a-by(t))\rd t+\gamma \sqrt{|y(t) (1-y(t))|}\rd W(t)
\end{align*}
with $a=\nu^{2}$, $b=\nu^{2}+2\kappa$ and $\gamma=2\nu$.
Moreover, by using the Lamperti transformation, $x(t):=2\arcsin(\sqrt{y(t)})=2\arcsin(X(t))$ is a solution of the SDE
\begin{align*}
\rd x(t)=f(x(t)) \rd t+\gamma \rd W(t),
\end{align*}
where $f(x):=(a-\gamma^{2}/4)\cot(x/2)-(b-a-\gamma^{2}/4)\tan(x/2)$.
Let $x^{(n)}=(x^{(n)}(t_{k}))_{k=0,1,\ldots,n}$ be the implicit Euler--Maruyama scheme for this SDE, that is, $x^{(n)}(t_{0})=x(0)$ and  $x^{(n)}(t_{k+1})=x^{(n)}(t_{k})+f(x^{(n)}(t_{k+1})) \Delta t+ \gamma \Delta_{k}W$, $k=0,1,\ldots,n-1$.
Then for any $p \in [2,\frac{4}{3\gamma^{2}}\max\{a,b-a\})$, it holds that
\begin{align*}
\e\left[
\max_{k=0,1,\ldots,n}
\left|
x(t_{k})-x^{(n)}(t_{k})
\right|^{p}
\right]^{1/p}
\leq
\frac{C_{p}}{n}
\end{align*}
(see Section 3.5 and Proposition 3.4 in \cite{NeSz} for more details.
Note that there it says $p \in [2,\frac{4}{3\gamma^{2}}\min\{a,b-a\})$ as the condition for $p$, but the correct condition is $p \in [2,\frac{4}{3\gamma^{2}}\max\{a,b-a\})$).
Therefore, since $\sin(x)$ is Lipschitz continuous, under the assumption $\kappa/\nu^{2}>3$ (i.e. $ 2<\frac{4}{3\gamma^{2}}\max\{a,b-a\}=\frac{2\kappa}{3\nu^{2}}$), it holds that
\begin{align*}
\e\left[
\max_{k=0,1,\ldots,n}
\left|
X(t_{k})-\sin\left(\frac{x^{(n)}(t_{k})}{2}\right)
\right|^{2}
\right]^{1/2}
\leq
\frac{C_{p}}{n}.
\end{align*}
\end{itemize}
\end{Rem}

By using Theorem \ref{thm:2.3}, we provide the $L^{2}$-rate of convergence for $\overline{X}^{(n)}$.
\begin{Cor}
\label{thm:2.1}
Suppose $\kappa/\nu^{2}>3$.
Then there exists $C>0$ such that for any $n \in \n$,
\begin{align*}
\max_{k=0,1,\ldots,n}{\mathbb E}\left[\left|X(t_{k})-\overline{X}^{(n)}(t_{k})\right|^{2}\right]^{1/2} \leq \frac{C}{n^{1/2}} \quad \text{ and } \quad {\mathbb E}\left[\max_{k=0,1,\ldots,n}\left|X(t_{k})-\overline{X}^{(n)}(t_{k})\right|^{2}\right]^{1/2} \leq \frac{C}{n^{1/4}}.
\end{align*}
\end{Cor}

\begin{proof}
By Theorem \ref{thm:2.3}, for any $k=0,1,\ldots,n$, it is suffice to estimate
$$
{\mathbb E}\left[\left|X^{(n)}(t_{k})-\overline{X}^{(n)}(t_{k})\right|^{2}\right] \quad \text{ and } \quad {\mathbb E}\left[\max_{k=0,1,\ldots,n}\left|X^{(n)}(t_{k})-\overline{X}^{(n)}(t_{k})\right|^{2}\right].
$$
For any $x \in {\mathbb R}^{d}$ and $y \in \mathscr{B}^{d}$, it holds that $|x-\Pi(x)| \leq |x-y|$.
Indeed, if $x \in \mathscr{B}^{d}$, then $|x-\Pi(x)|=0 \leq |x-y|$, and if $x \in {\mathbb R}^{d} \setminus \mathscr{B}^{d}$, then we obtain $|x-\Pi(x)|=|x|-1 \leq |x|-|y| \leq |x-y|$.
Hence it holds that
$$
\left|X^{(n)}(t_{k})-\overline{X}^{(n)}(t_{k})\right| \leq \left|X^{(n)}(t_{k})-X(t_{k})\right|.
$$
This estimate together with Theorem \ref{thm:2.3} yields the statements.
\end{proof}

Before we prove Theorem \ref{thm:2.3}, we study some properties of the solution $Y=(Y_{0},X_{1},\ldots,X_{d})^{\top}$ of the system of SDE \eqref{eq:4}.
In the following proposition, we estimate the inverse moment of $Y_{0}$ and the Kolmogorov type condition of $Y$.

\begin{Prop}\label{lem:2.2}
Suppose that $\kappa/\nu^{2} \geq 1$ and let $q \in (0,2\kappa/\nu^{2})$.
\begin{itemize}
\item[(i)]
There exists $C_{1}(q)>0$ such that
$$
\sup_{0 \leq t \leq T}{\mathbb E}\left[Y_{0}(t)^{-q}\right] \leq C_{1}(q).
$$
\item[(ii)]
There exists $C_{2}(q)>0$ such that for any $s,t\in [0,T]$,
$$
\max_{i=0,1,\ldots,d}{\mathbb E}\left[|Y_{i}(t)-Y_{i}(s)|^{q}\right] \leq C_{2}(q)|t-s|^{q/2}.
$$
\end{itemize}
\end{Prop}

In order to prove Proposition \ref{lem:2.2}, we consider the following one-dimensional SDE called the Wright--Fisher diffusion with positive parameters $a$, $b$ and $\gamma$:
\begin{align}
\label{eq:9}
\rd y(t)=(a-by(t))\rd t+\gamma \sqrt{|y(t) (1-y(t))|}\rd \widetilde{W}(t),~y(0) \in [0,1],
\end{align}
where $\widetilde{W}=(\widetilde{W}(t))_{t \geq 0}$ is a standard one-dimensional Brownian motion.
It is well-known that the Wright--Fisher diffusion \eqref{eq:9} has a unique strong solution which takes values in $[0,1]$ if and only if $2a/\gamma^{2} \geq 1$ and $2(b-a)/\gamma^{2} \geq 1$.

The following lemma gives the inverse moment estimate of the Wright--Fisher diffusion \eqref{eq:9}.

\begin{Lem}[e.g. Section 3.5 in \cite{NeSz} and Section 4 in \cite{HuKu08}]
\label{lem:1.5}
Let $y=(y(t))_{t \geq 0}$ be a Wright--Fisher diffusion \eqref{eq:9} with the initial condition $y(0) \in (0,1)$ and positive parameters $a$, $b$ and $\gamma$.
Assume that $2a/\gamma^{2} \geq 1$.
Then for any $0<q<2a/\gamma^{2}$, there exists $C(q,y(0))>0$ such that
\begin{align*}
\sup_{0 \leq t \leq T}
{\mathbb E}[y(t)^{-q}]
\leq
C(q,y(0))
\end{align*}
\end{Lem}

By using this lemma, we prove Proposition \ref{lem:2.2}.

\begin{proof}[Proof of Proposition \ref{lem:2.2}]
Proof of (i).
We first show that $y:=1-|X|^{2}=(1-|X(t)|^{2})_{t \in [0,T]}$ is a solution of the Wright--Fisher diffusion \eqref{eq:9} with some parameters.
Define a new one-dimensional Brownian motion $\widetilde{W}=(\widetilde{W}(t))_{t \in [0,T]}$ by
\begin{align*}
\widetilde{W}(t):=-\sum_{i=1}^{d}\int_{0}^{t}g_{i}(X(s))\rd W_{i}(s),
\end{align*}
where $g_{i}: \mathscr{B}^{d} \to {\mathbb R}$, $i=1,\ldots,d$ are defined by
\begin{align*}
g_{i}(x)=
\left\{
\begin{array}{ll}
\displaystyle{\frac{x_{i}}{|x|}} & \text{ if } |x| \neq 0, \\
\displaystyle{\frac{1}{\sqrt[]{d}}} & \text{ if } |x|=0.
\end{array}
\right.
\end{align*}
It follows from L\'evy's theorem (e.g. Theorem 3.3.16 in \cite{KS}) that $\widetilde{W}$ is a standard one-dimensional Brownian motion.
Then by \eqref{eq:12}, we see that $y=1-|X|^{2}$ is a Wright--Fisher diffusion \eqref{eq:9} driven by $\widetilde{W}$ with the initial condition $1-|x(0)|^{2} \in (0,1]$, $a=2\kappa$, $b=d\nu^{2}+2\kappa$ and $\gamma=2\nu$.


We first assume $|x(0)| \in (0,1)$.
Then it holds from Lemma \ref{lem:1.5} with $0<q/2<\kappa/\nu^{2}(=2a/\gamma^{2})$ that
\begin{align*}
\sup_{0 \leq t \leq T}
{\mathbb E}[Y_{0}(t)^{-q}]
=
\sup_{0 \leq t \leq T}
{\mathbb E}
\left[(1-|X(t)|^{2})^{-q/2}\right]
=
\sup_{0 \leq t \leq T}
{\mathbb E}
\left[y(t)^{-q/2}\right]
\leq C(q/2,1-|x(0)|^{2}).
\end{align*}
This concludes the statement (i) for $|x(0)| \in (0,1)$.
Now we assume $|x(0)|=0$ (i.e. $y(0)=1$).
Then by using the comparison theorem (see e.g. Proposition 5.2.18 in \cite{KS}), it holds from Lemma \ref{lem:1.5} with $0<q/2<\kappa/\nu^{2}(=2a/\gamma^{2})$ that
\begin{align*}
\sup_{0 \leq t \leq T}
{\mathbb E}[Y_{0}(t)^{-q}]
&=
\sup_{0 \leq t \leq T}
{\mathbb E}
\left[(1-|X(t)|^{2})^{-q/2}\right]
=
\sup_{0 \leq t \leq T}
{\mathbb E}
\left[y(t)^{-q/2}\right]
\\&\leq
\sup_{0 \leq t \leq T}
{\mathbb E}
\left[y_{1/2}(t)^{-q/2}\right]
\leq C(q/2,1/2),
\end{align*}
where $y_{1/2}$ is a Wright--Fisher diffusion \eqref{eq:9} with the initial condition $y_{1/2}(0)=1/2 \in (0,1)$, $a=2\kappa$, $b=d\nu^{2}+2\kappa$ and $\gamma=2\nu$.
This concludes the statement (i) for $|x(0)|=0$.


Proof of (ii).
Let $0 \leq s \leq t \leq T$ be fixed.
By using Burkholder--Davis--Gundy's inequality, we obtain
$$
{\mathbb E}\left[|Y_{0}(t)-Y_{0}(s)|^{q}\right] \leq C_{q}\left\{\int_{s}^{t}{\mathbb E}\left[Y_{0}(u)^{-q}\right]{\rm d}u+{\mathbb E}\left[\left|\int_{s}^{t}Y_{0}(u){\rm d}u\right|^{q}\right]+\sum_{i=1}^{d}{\mathbb E}\left[\left|\int_{s}^{t}X_{i}(u)^{2}{\rm d}u\right|^{q/2}\right]\right\}
$$
for some constant $C_{q}>0$.
Thus noting that $Y_{0} \in (0,1]$ and $|X| \in [0,1)$ a.s., from the statement (i), we have the Kolmogorov type condition for $Y_{0}$.
Next, for any $i=1,\ldots,d$, by using Burkholder--Davis--Gundy's inequality again, we obtain
\begin{align*}
&{\mathbb E}\left[|X_{i}(t)-X_{i}(s)|^{q}\right] \\
&\leq
C_{q}
\left\{
{\mathbb E}\left[
\left|
\int_{s}^{t}
|X_{i}(u)|
{\rm d}u
\right|^{q}
\right]
+
{\mathbb E}
\left[
\left|
\int_{s}^{t}
Y_{0}(u)^{2}
{\rm d}u
\right|^{q/2}
\right]
+
{\mathbb E}
\left[
\left|
\int_{s}^{t}
|(A_{0}X(u))_{i}|
{\rm d}u\right.
\right|^{q}
\right]
\\
&\hspace{1.12cm}\left.
+
\sum_{p=1}^{m}
{\mathbb E}
\left[
\left|
\int_{s}^{t}(A_{p}X(u))_{i}^{2}{\rm d}u
\right|^{q/2}
\right]
+
\sum_{p=1}^{m}
{\mathbb E}
\left[
\left|
\int_{s}^{t}
|(A_{p}^{2}X(u))_{i}|
{\rm d}u
\right|^{q}
\right]
\right\}
\end{align*}
for some constant $C_{q}>0$.
Here noting that $|X| \in [0,1)$ a.s., by using Cauchy--Schwarz's inequality, it holds that for any $d \times d$ real matrix $A$,
$$
(AX(u))_{i}^{2} \leq \|A\|^{2} \quad \text{a.s.}
$$
Hence noting that $|X| \in [0,1)$ and $Y_{0} \in (0,1]$ a.s., we obtain the Kolmogorov type condition for $X=(Y_{1},\ldots,Y_{d})$.
\end{proof}

By using the estimates in Proposition \ref{lem:2.2}, we prove Theorem \ref{thm:2.3}.


\begin{proof}[Proof of Theomre \ref{thm:2.3}]
We first decompose $Y=(Y_{0},X_{1},\ldots,X_{d})^{\top}$.
Fix $k=0,1,\ldots,n-1$ and we define $R^{M}(k)=(R_{0}^{M}(k),\ldots,R_{d}^{M}(k))^{\top}$ and $R^{A}(k)=(R_{0}^{A}(k),\ldots,R_{d}^{A}(k))^{\top}$ by
\begin{align*}
R_{0}^{M}(k)&:=-\nu \int_{t_{k}}^{t_{k+1}} \{X(s)-X(t_{k})\}^{\top} \rd W(s) \\
R_{0}^{A}(k)&:=\left(\kappa-\frac{\nu^{2}}{2}\right)\int_{t_{k}}^{t_{k+1}}\left\{\frac{1}{Y_{0}(s)}-\frac{1}{Y_{0}(t_{k+1})}\right\}\rd s
-\left(\kappa-\frac{\nu^{2}}{2}+\frac{d\nu^{2}}{2}\right)\int_{t_{k}}^{t_{k+1}}\{Y_{0}(s)-Y_{0}(t_{k})\}\rd s
\end{align*}
and for any $i=1,\ldots,d$,
\begin{align*}
R_{i}^{M}(k)&:=\nu\int_{t_{k}}^{t_{k+1}}\{Y_{0}(s)-Y_{0}(t_{k})\}\rd W_{i}(s)+\sum_{p=1}^{m}\int_{t_{k}}^{t_{k+1}}(A_{p}\{X(s)-X(t_{k})\})_{i}\rd \widehat{W}_{p}(s), \\
R_{i}^{A}(k)&:=-\kappa\int_{t_{k}}^{t_{k+1}}\{X(s)-X(t_{k})\}_{i}\rd s+\int_{t_{k}}^{t_{k+1}}(A_{0}\{X(s)-X(t_{k})\})_{i}\rd s \\
&\hspace{0.5cm}+\frac{1}{2}\sum_{p=1}^{m}\int_{t_{k}}^{t_{k+1}}(A_{p}^{2}\{X(s)-X(t_{k})\})_{i}\rd s.
\end{align*}
Then we obtain the decomposition of $Y=(Y_{0},X_{1},\ldots,X_{d})^{\top}$ as follows:
\begin{align}
\label{eq:14}
\begin{split}
Y_{0}(t_{k+1})&=Y_{0}(t_{k})-\nu X(t_{k})^{\top}\Delta_{k}W+\frac{\kappa-\frac{\nu^{2}}{2}}{Y_{0}(t_{k+1})}\Delta t \\
&\hspace{0.4cm}-\left(\kappa-\frac{\nu^{2}}{2}+\frac{d\nu^{2}}{2}\right)Y_{0}(t_{k})\Delta t+R_{0}^{M}(k)+R_{0}^{A}(k)
\end{split}
\end{align}
and
\begin{align}
\label{eq:15}
\begin{split}
X_{i}(t_{k+1})&=X_{i}(t_{k})-\kappa X_{i}(t_{k})\Delta t+\nu Y_{0}(t_{k})\Delta_{k}W_{i}+(A_{0}X(t_{k}))_{i}\Delta t \\
&\hspace{0.4cm}+\sum_{p=1}^{m}(A_{p}X(t_{k}))_{i}\Delta_{k}\widehat{W}_{p}+\frac{1}{2}\sum_{p=1}^{m}(A_{p}^{2}X(t_{k}))_{i}\Delta t+R_{i}^{M}(k)+R_{i}^{A}(k).
\end{split}
\end{align}
Moreover, we define $S^{M}(k)=(S_{0}^{M}(k),\ldots,S_{d}^{M}(k))^{\top}$ and $S^{A}(k)=(S_{0}^{A}(k),\ldots,S_{d}^{A}(k))^{\top}$ by
\begin{align*}
S_{0}^{M}(k)&:=-\nu \langle X(t_{k})-X^{(n)}(t_{k}), \Delta_{k}W \rangle,\\
S_{0}^{A}(k)&:=-\left(\kappa-\frac{\nu^{2}}{2}+\frac{d\nu^{2}}{2}\right)\{Y_{0}(t_{k})-Y_{0}^{(n)}(t_{k})\}\Delta t
\end{align*}
and for any $i=1,\ldots,d$,
\begin{align*}
S_{i}^{M}(k)&:=\nu\{Y_{0}(t_{k})-Y_{0}^{(n)}(t_{k})\}\Delta_{k}W_{i}+\sum_{p=1}^{m}(A_{p}\{X(t_{k})-X^{(n)}(t_{k})\})_{i}\Delta_{k}\widehat{W}_{p}, \\
S_{i}^{A}(k)&:=-\kappa\{X(t_{k})-X^{(n)}(t_{k})\}_{i}\Delta t+(A_{0}\{X(t_{k})-X^{(n)}(t_{k})\})_{i}\Delta t \\
&\hspace{0.5cm}+\frac{1}{2}\sum_{p=1}^{m}(A_{p}^{2}\{X(t_{k})-X^{(n)}(t_{k})\})_{i}\Delta t,
\end{align*}
and we set $e(k):=Y(t_{k})-Y^{(n)}(t_{k})$ and $r(k):=R^{M}(k)+R^{A}(k)+S^{M}(k)+S^{A}(k)$.
Then by \eqref{eq:14}, \eqref{eq:15} and the definition of $Y^{(n)}=(Y_{0}^{(n)},X_{1}^{(n)},\ldots,X_{d}^{(n)})^{\top}$, we have
\begin{align*}
e_{0}(k+1)=e_{0}(k)+r_{0}(k)+\left(\kappa-\frac{\nu^{2}}{2}\right)\left\{\frac{1}{Y_{0}(t_{k+1})}-\frac{1}{Y_{0}^{(n)}(t_{k+1})}\right\}\Delta t
\end{align*}
and for any $i=1,\ldots,d$,
\begin{align*}
e_{i}(k+1)=e_{i}(k)+r_{i}(k).
\end{align*}
By the definition of $(e(k))_{k=0,1,\ldots,n}$, we obtain
\begin{align*}
&|e(k+1)|^{2}-2\left(\kappa-\frac{\nu^{2}}{2}\right)e_{0}(k+1)\left\{\frac{1}{Y_{0}(t_{k+1})}-\frac{1}{Y_{0}^{(n)}(t_{k+1})}\right\}\Delta t \\
&\leq \left\{e_{0}(k+1)-\left(\kappa-\frac{\nu^{2}}{2}\right)\left\{\frac{1}{Y_{0}(t_{k+1})}-\frac{1}{Y_{0}^{(n)}(t_{k+1})}\right\}\Delta t\right\}^{2}+\sum_{i=1}^{d}e_{i}(k+1)^{2} \\
&=|e(k)+r(k)|^{2}.
\end{align*}
Thus by the assumption $\kappa/\nu^{2} > 3$ (this implies $\kappa-\nu^{2}/2>0$), we have
\begin{align*}
|e(k+1)|^{2} &\leq |e(k)+r(k)|^{2}+2\left(\kappa-\frac{\nu^{2}}{2}\right)e_{0}(k+1)\left\{\frac{1}{Y_{0}(t_{k+1})}-\frac{1}{Y_{0}^{(n)}(t_{k+1})}\right\}\Delta t \\
&\leq |e(k)+r(k)|^{2}=|e(k)|^{2}+2\langle e(k),r(k) \rangle+|r(k)|^{2},
\end{align*}
where we used the fact that $(x-y)(1/x-1/y)=-(x-y)^{2}/xy \leq 0$, $x,y>0$ in the second inequality.
Hence we obtain
\begin{align*}
|e(k)|^{2}=\sum_{\ell=0}^{k-1}\left\{|e(\ell+1)|^{2}-|e(\ell)|^{2}\right\}
\leq \sum_{\ell=0}^{k-1}\left\{2\langle e(\ell),r(\ell) \rangle+|r(\ell)|^{2}\right\}.
\end{align*}
Here by using the inequality $\langle x,y \rangle \leq |x|^{2}/(2\delta)+\delta |y|^{2}/2$ for $\delta>0$,
\begin{align*}
\langle e(\ell),R^{A}(\ell) \rangle \leq \frac{1}{2}|R^{A}(\ell)|^{2}\frac{1}{\Delta t}+\frac{1}{2}|e(\ell)|^{2}\Delta t,
\end{align*}
and by the assumption $\kappa/\nu^{2}>3$ (this implies $\kappa-\nu^{2}/2>0$) and \eqref{eq:3},
\begin{align*}
\langle e(\ell),S^{A}(\ell) \rangle&=-\left(\kappa-\frac{\nu^{2}}{2}+\frac{d\nu^{2}}{2}\right)e_{0}(k)^{2}\Delta t-\kappa\sum_{i=1}^{d}e_{i}(\ell)^{2}\Delta t \\
&\hspace{0.4cm}+\langle A_{0}\{X(t_{\ell})-X^{(n)}(t_{\ell})\}, X(t_{\ell})-X^{(n)}(t_{\ell}) \rangle \Delta t \\
&\hspace{0.4cm}+\frac{1}{2}\sum_{p=1}^{m} \langle A_{p}^{2}\{X(t_{\ell})-X^{(n)}(t_{\ell})\}, X(t_{\ell})-X^{(n)}(t_{\ell}) \rangle \Delta t \\
&\leq 0.
\end{align*}
Therefore, it holds that
\begin{align}
\label{eq:16}
|e(k)|^{2} \leq 2\sum_{\ell=0}^{k-1}\langle e(\ell),R^{M}(\ell)+S^{M}(\ell) \rangle+\sum_{\ell=0}^{k-1}|R^{A}(\ell)|^{2}\frac{1}{\Delta t}+\frac{1}{2}\sum_{\ell=0}^{k-1}|e(\ell)|^{2}\Delta t+\sum_{\ell=0}^{k-1}|r(\ell)|^{2}.
\end{align}

Fix $\ell=0,1,\ldots,k-1$.
We first estimate the expectations of the terms $\langle e(\ell),R^{M}(\ell)+S^{M}(\ell) \rangle$ and $|R^{A}(\ell)|^{2}$ on the right hand side of \eqref{eq:16}.
Noting that $e(\ell)$ is ${\mathcal F}(t_{\ell})$-measurable, by using the martingale property of the Brownian motion $B=(W_{1},\ldots,W_{d},\widehat{W}_{1},\ldots,\widehat{W}_{m})^{\top}$ and stochastic integrals, we have
\begin{align}
\label{eq:17}
{\mathbb E}\left[\langle e(\ell),R^{M}(\ell)+S^{M}(\ell) \rangle\right]=\sum_{i=0}^{d}{\mathbb E}\left[e_{i}(\ell){\mathbb E}\left[\left.R_{i}^{M}(\ell)+S_{i}^{M}(\ell) \,\right|\, {\mathcal F}(t_{\ell})\right]\right]
=0.
\end{align}
By using Jensen's inequality, we obtain
\begin{align*}
{\mathbb E}\left[|R^{A}(\ell)|^{2}\right]
&\leq 2\left\{\left(\kappa-\frac{\nu^{2}}{2}\right)^{2}\int_{t_{\ell}}^{t_{\ell+1}}{\mathbb E}\left[\left\{\frac{1}{Y_{0}(s)}-\frac{1}{Y_{0}(t_{\ell+1})}\right\}^{2}\right]\rd s\right. \\ \notag
&\hspace{0.9cm}\left.+\left(\kappa-\frac{\nu^{2}}{2}+\frac{d\nu^{2}}{2}\right)^{2}\int_{t_{\ell}}^{t_{\ell+1}}{\mathbb E}\left[\{Y_{0}(s)-Y_{0}(t_{\ell})\}^{2}\right]\rd s\right\}\Delta t \\ \notag
&\hspace{0.4cm}+(m+2)\Biggl\{\kappa^{2}\int_{t_{\ell}}^{t_{\ell+1}}{\mathbb E}\left[|X(s)-X(t_{\ell})|^{2}\right]\rd s+\int_{t_{\ell}}^{t_{\ell+1}}{\mathbb E}\left[|A_{0}\{X(s)-X(t_{\ell})\}|^{2}\right]\rd s \\ \notag
&\hspace{2.2cm}+\frac{1}{4}\sum_{p=1}^{m}\int_{t_{\ell}}^{t_{\ell+1}}{\mathbb E}\left[|A_{p}^{2}\{X(s)-X(t_{\ell})\}|^{2}\right]\rd s\Biggl\}\Delta t.
\end{align*}
Here by using H\"older's inequality, and Proposition \ref{lem:2.2} (i) and (ii) with $q=6<2 \kappa/\nu^{2}$, for any $s \in [t_{\ell},t_{\ell+1}]$,
\begin{align}
{\mathbb E}\left[\left\{\frac{1}{Y_{0}(s)}-\frac{1}{Y_{0}(t_{\ell+1})}\right\}^{2}\right]
&\leq {\mathbb E}\left[|Y_{0}(t_{\ell+1})-Y_{0}(s)|^{6}\right]^{1/3}{\mathbb E}\left[Y_{0}(s)^{-6}\right]^{1/3}{\mathbb E}\left[Y_{0}(t_{\ell+1})^{-6}\right]^{1/3} 
\notag \\
&\leq C_{1}(6)^{2/3} C_{2}(6)^{1/3} \Delta t, \label{eq:17.1}
\end{align}
and by Cauchy--Schwarz's inequality and Proposition \ref{lem:2.2} (ii), for any $d \times d$ real matrix $A$,
\begin{align}
\label{eq:18}
{\mathbb E}\left[|A\{X(s)-X(t_{\ell})\}|^{2}\right] &\leq \|A\|^{2}{\mathbb E}\left[|X(s)-X(t_{\ell})|^{2}\right]
\leq \|A\|^{2} d C_{2}(2)\Delta t.
\end{align}
Thus we have
\begin{align}
\label{eq:19}
{\mathbb E}\left[|R^{A}(\ell)|^{2}\right] 
&\leq
C_{1}(\Delta t)^{3}
\end{align}
for some constant $C_{1}>0$.

We next estimate the expectation of the term $|r(\ell)|^{2}$ on the right hand side of \eqref{eq:16}.
By It\^o's isometry of stochastic integrals, we obtain
\begin{align*}
{\mathbb E}\left[|R^{M}(\ell)|^{2}\right]&=\nu^{2}\int_{t_{\ell}}^{t_{\ell+1}}{\mathbb E}\left[|X(s)-X(t_{\ell})|^{2}\right]\rd s+d\nu^{2}\int_{t_{\ell}}^{t_{\ell+1}}{\mathbb E}\left[\{Y_{0}(s)-Y_{0}(t_{\ell})\}^{2}\right]\rd s \\ \notag
&\hspace{0.4cm}+\sum_{p=1}^{m}\int_{t_{\ell}}^{t_{\ell+1}}{\mathbb E}\left[|A_{p}\{X(s)-X(t_{\ell})\}|^{2}\right]\rd s.
\end{align*}
Thus by Proposition \ref{lem:2.2} (ii), we have
\begin{align}
\label{eq:20}
{\mathbb E}\left[|R^{M}(\ell)|^{2}\right]
\leq
C_{2}(\Delta t)^{2}
\end{align}
for some constant $C_{2}>0$.
By using the estimate \eqref{eq:18}, we obtain
\begin{align}
\label{eq:21}
{\mathbb E}\left[|S^{M}(\ell)|^{2}\right]
&=
\nu^{2}{\mathbb E}\left[|X(t_{\ell})-X^{(n)}(t_{\ell})|^{2}\right]\Delta t
+
d\nu^{2}{\mathbb E}\left[\{Y_{0}(t_{\ell})-Y_{0}^{(n)}(t_{\ell})\}^{2}\right]\Delta t 
\notag
\\
&\hspace{0.4cm}
+\sum_{p=1}^{m}{\mathbb E}\left[|A_{p}\{X(t_{\ell})-X^{(n)}(t_{\ell})\}|^{2}\right]\Delta t \notag\\
&\leq 
C_{3}{\mathbb E}\left[|e(\ell)|^{2}\right]\Delta t
\end{align}
for some constant $C_{3}>0$, and by the estimate \eqref{eq:18} again, we have
\begin{align}
\label{eq:22}
{\mathbb E}\left[|S^{A}(\ell)|^{2}\right]
&\leq
\left(\kappa-\frac{\nu^{2}}{2}+\frac{d\nu^{2}}{2}\right)^{2}
{\mathbb E}\left[\{Y_{0}(t_{\ell})-Y_{0}^{(n)}(t_{\ell})\}^{2}\right](\Delta t)^{2} \notag\\
&\hspace{0.4cm}
+
(m+2)
\Big\{
	\kappa^{2}{\mathbb E}\left[|X(t_{\ell})-X^{(n)}(t_{\ell})|^{2}\right]+{\mathbb E}\left[|A_{0}\{X(t_{\ell})-X^{(n)}(t_{\ell})\}|^{2}\right] \notag\\
	&\hspace{2.25cm}
	+\frac{1}{4}\sum_{p=1}^{m}{\mathbb E}\left[|A_{p}^{2}\{X(t_{\ell})-X^{(n)}(t_{\ell})\}|^{2}\right]
\Big\}(\Delta t)^{2} \notag\\
&\leq 
C_{4}{\mathbb E}\left[|e(\ell)|^{2}\right](\Delta t)^{2}
\end{align}
for some constant $C_{4}>0$.
By combining \eqref{eq:19}, \eqref{eq:20}, \eqref{eq:21} and \eqref{eq:22}, it holds that
\begin{align}
\label{eq:23}
{\mathbb E}\left[|r(\ell)|^{2}\right]
&\leq
4\left\{
{\mathbb E}\left[|R^{M}(\ell)|^{2}\right]
+{\mathbb E}\left[|R^{A}(\ell)|^{2}\right]
+{\mathbb E}\left[|S^{M}(\ell)|^{2}\right]
+{\mathbb E}\left[|S^{A}(\ell)|^{2}\right]
\right\} \notag \\
&\leq 4\left\{(C_{1}T+C_{2})(\Delta t)^{2}+(C_{3}+C_{4}T){\mathbb E}\left[|e(\ell)|^{2}\right]\Delta t\right\}.
\end{align}
Therefore, if follows from \eqref{eq:16}, \eqref{eq:17}, \eqref{eq:19} and \eqref{eq:23} that
\begin{align}
\label{eq:24}
{\mathbb E}\left[|e(k)|^{2}\right] 
&\leq
\sum_{\ell=0}^{k-1}
\left\{
	{\mathbb E}
	\left[|R^{A}(\ell)|^{2}\right]
	\frac{1}{\Delta t}
	+
	{\mathbb E}
	\left[|e(\ell)|^{2}\right]
	\Delta t
	+
	{\mathbb E}
	\left[|r(\ell)|^{2}\right]
\right\} \notag\\
&\leq
(C_{1}+4(C_{1}T+C_{2}))\Delta t
+
(1+4(C_{3}+C_{4}T))
\sum_{\ell=0}^{k-1}
{\mathbb E}
\left[|e(\ell)|^{2}\right]
\Delta t \notag \\
&=
C_{5}\Delta t
+
C_{6}
\sum_{\ell=0}^{k-1}
{\mathbb E}
\left[|e(\ell)|^{2}\right]
\Delta t,
\end{align}
where $C_{5}:=(C_{1}+4(C_{1}T+C_{2}))$ and $C_{6}:=(1+4(C_{3}+C_{4}T))$.
Hence by using the discrete Gronwall inequality, we obtain
\begin{align}
\label{eq:25}
\max_{k=0,1,\ldots,n}
{\mathbb E}\left[|e(k)|^{2}\right] \leq C_{5}\left(1+C_{6}e^{C_{6}}\right)\Delta t
=:C_{7}\Delta t,
\end{align}
which concludes the first statement.

In order to prove the second statement, we need to estimate the expectation of the random variable
\begin{align*}
\sum_{\ell=0}^{n-1}
\left|\langle e(\ell),R^{M}(\ell)+S^{M}(\ell) \rangle \right|.
\end{align*}
By using Cauchy--Schwarz's inequality, \eqref{eq:20}, \eqref{eq:21} and \eqref{eq:25}, we have
\begin{align*}
{\mathbb E}\left[\left|\langle e(\ell),R^{M}(\ell)+S^{M}(\ell) \rangle\right|\right] &\leq {\mathbb E}\left[|e(\ell)|^{2}\right]^{1/2}\left\{{\mathbb E}\left[|R^{M}(\ell)|^{2}\right]^{1/2}+{\mathbb E}\left[|S^{M}(\ell)|^{2}\right]^{1/2}\right\} \\
&\leq {\mathbb E}\left[|e(\ell)|^{2}\right]^{1/2}\left\{C_{2}^{1/2}\Delta t+C_{3}^{1/2}{\mathbb E}\left[|e(\ell)|^{2}\right]^{1/2}(\Delta t)^{1/2}\right\} \\
&\leq
C_{8}(\Delta t)^{3/2}
\end{align*}
for some constant $C_{8}>0$.
Therefore, it follows from \eqref{eq:24} and \eqref{eq:25} that
\begin{align*}
{\mathbb E}\left[\max_{k=0,1,\ldots,n}|e(k)|^{2}\right] &\leq 2\sum_{\ell=0}^{n-1}{\mathbb E}\left[|e(\ell)|\left\{|R^{M}(\ell)|+|S^{M}(\ell)|\right\}\right]+C_{5}\Delta t+C_{6}\sum_{\ell=0}^{k-1}{\mathbb E}\left[|e(\ell)|^{2}\right]\Delta t \\
&\leq
C_{9}
\Big\{
	(\Delta t)^{1/2}
	+\sum_{\ell=0}^{n-1}{\mathbb E}\left[|e(\ell)|^{2}\right]\Delta t
\Big\}
\leq
C_{9}
\Big\{
	(\Delta t)^{1/2}
	+
	C_{7}
	\Delta t
\Big\},
\end{align*}
which concludes the second statement.
\end{proof}

\section{Numerical experiments}
\label{sec:3}
In this section, we provide some numerical results about the projection scheme \eqref{eq:8} for the polynomial diffusions \eqref{SDE_ball_0} and \eqref{eq:1}, and about the difference between two solutions of the system of the SDE \eqref{eq:4} using the projection scheme.

First, we observe behaviors of the projection scheme \eqref{eq:8} for the polynomial diffusion \eqref{SDE_ball_0} with respect to the parameter $\kappa$ through numerical experiments.
Figures \ref{fig_1}, \ref{fig_2}, \ref{fig_3}, \ref{fig_4} describe sample paths of the projection scheme $\overline{X}^{(n)}$ with $n=10000$ time steps for the polynomial diffusion $\rd X(t)=-\kappa X(t)\rd t+\nu \sqrt{1-|X(t)|^{2}} \rd W(t)$ on the time interval $[0,T]$ with $d=2$, $T=1$, $x(0)=(0.7,0.7)^{\top}$, $\nu=\sqrt{2}$ and $\kappa=2,7,50,100$.
Here $\kappa=2$ (resp. $\kappa=7$) means the smallest natural number for which the pathwise uniqueness (resp. the assumption of Corollary \ref{thm:2.3}) holds.
Moreover, the color gradient in the figures represents the flow of time.
The behaviors of $\overline{X}^{(n)}$ in the figures can be theoretically explained as follows.
By using \eqref{eq:12} and Gronwall's inequality, it holds that
\begin{align*}
\e\left[|X(t)|^{2}\right]
\leq
(|x(0)|^{2}+d\nu^{2} t)
e^{-(d\nu^{2}+2\kappa)t}
\to 0,~t \to \infty.
\end{align*}
Hence for any (small) $\varepsilon>0$, by using Markov's inequality, it holds that
\begin{align*}
\p(|X(t)|<\varepsilon)\geq 1-\varepsilon^{-2}(|x(0)|^{2}+d\nu^{2} t) e^{-(d\nu^{2}+2\kappa)t} \to 1,~t \to \infty.
\end{align*}
Therefore if $\kappa$ and $t$ are sufficiently large, then the value of $|X(t)|$ is small with high probability.
Figure \ref{fig_1}, \ref{fig_2}, \ref{fig_3}, \ref{fig_4} express this fact.

\begin{figure}[t!]
\begin{minipage}{0.5\hsize}
\includegraphics[width=80mm]
{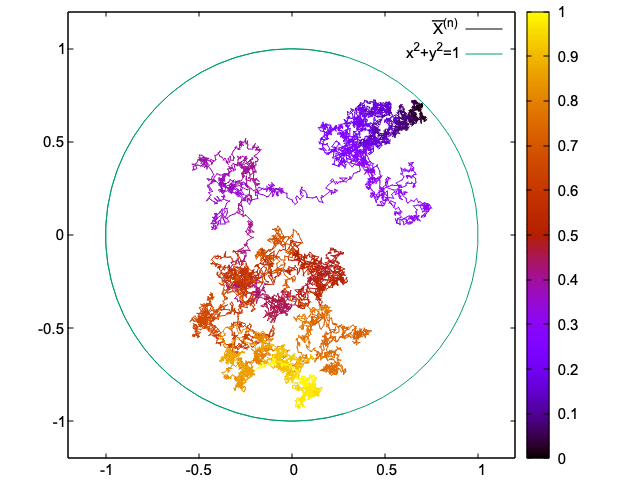}
\caption{$\kappa=2$.}
\label{fig_1}
\end{minipage}
\begin{minipage}{0.5\hsize}
\includegraphics[width=80mm]
{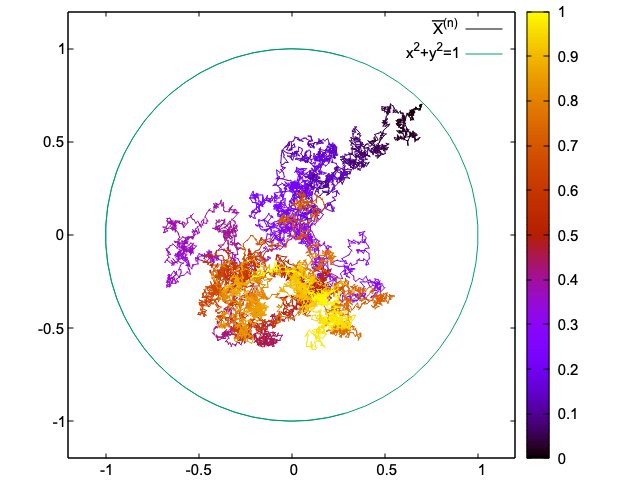}
\caption{$\kappa=7$.}
\label{fig_2}
\end{minipage}
\begin{minipage}{0.5\hsize}
\includegraphics[width=80mm]
{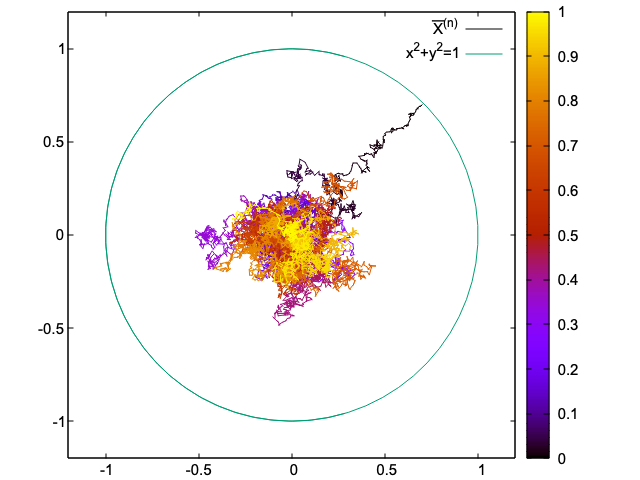}
\caption{$\kappa=50$.}
\label{fig_3}
\end{minipage}
\begin{minipage}{0.5\hsize}
\includegraphics[width=80mm]
{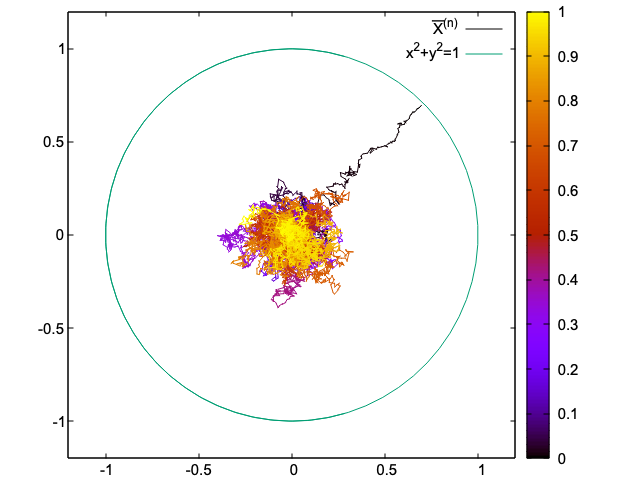}
\caption{$\kappa=100$.}
\label{fig_4}
\end{minipage}
\end{figure}

Next, we observe behaviors of the projection scheme \eqref{eq:8} for the polynomial diffusion \eqref{eq:1} with respect to real skew symmetric matrices $A_{p}$, $p=0,1,\ldots,m$ through numerical experiments.
Figures \ref{fig_5}, \ref{fig_6}, \ref{fig_7}, \ref{fig_8} describe sample paths of the projection scheme $\overline{X}^{(n)}$ with $n=10000$ time steps for the polynomial diffusion $\rd X(t)=-\kappa X(t)\rd t +\nu \sqrt{1-|X(t)|^{2}} \rd W(t)+A_{0}X(t)\rd t +A_{1}X(t) \circ \rd \widehat{W}_{1}(t)$ on the time interval $[0,T]$ with $d=2$, $m=1$, $T=1$, $x(0)=(0.7,0.7)^{\top}$, $\nu=\sqrt{2}$, $\kappa=7$ and various real skew symmetric matrices $A_{0}$ and $A_{1}$.
Note that the real symmetric matrix with a large value of the Frobenius norm has a strong effect on the direction of the corresponding vector.
For example, in Figures \ref{fig_5}, \ref{fig_6}, \ref{fig_7}, \ref{fig_8}, if the Frobenius norm of the matrix $A_{1}^{2}$ is sufficiently large, we can see the matrix has a strong effect on the direction of the vector $(1/2) A_{1}^{2} x$ (recall $A_{1}X(t) \circ \rd \widehat{W}_{1}(t)=A_{1}X(t) \rd \widehat{W}_{1}(t)+(1/2)A_{1}^{2}X(t) \rd t$).



\begin{figure}[t!]
\begin{minipage}{0.5\hsize}
\savebox{\smlmat}{$A_{0}=\begin{pmatrix}0 & 1 \\ -1 & 0 \end{pmatrix}$,
$A_{1}=\begin{pmatrix}0 & 0 \\ 0 & 0 \end{pmatrix}$}
\includegraphics[width=80mm]
{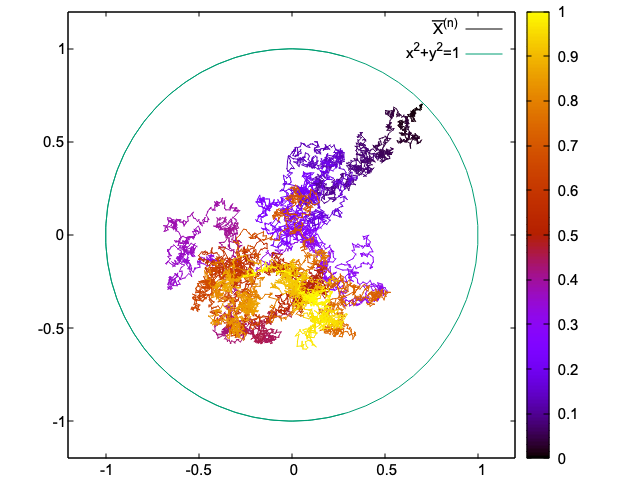}
\caption{\usebox{\smlmat}.}
\label{fig_5}
\end{minipage}
\begin{minipage}{0.5\hsize}
\savebox{\smlmat}{$A_{0}=\begin{pmatrix}0 & 10 \\ -10 & 0 \end{pmatrix}$,
$A_{1}=\begin{pmatrix}0 & 0 \\ 0 & 0 \end{pmatrix}$}
\includegraphics[width=80mm]
{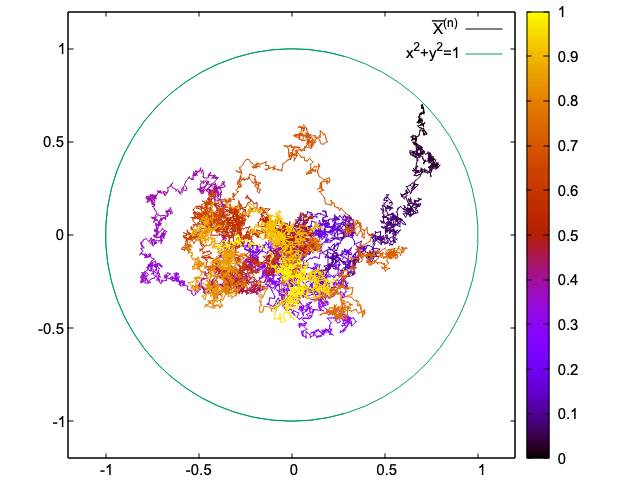}
\caption{\usebox{\smlmat}.}
\label{fig_6}
\end{minipage}
\begin{minipage}{0.5\hsize}
\savebox{\smlmat}{$A_{0}=\begin{pmatrix}0 & 1 \\ -1 & 0 \end{pmatrix}$,
$A_{1}=\begin{pmatrix}0 & 1 \\ -1 & 0 \end{pmatrix}$}
\includegraphics[width=80mm]
{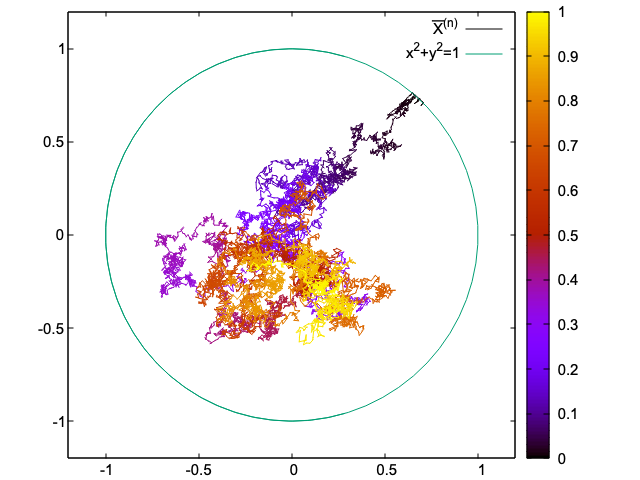}
\caption{\usebox{\smlmat}.}
\label{fig_7}
\end{minipage}
\begin{minipage}{0.5\hsize}
\savebox{\smlmat}{$A_{0}=\begin{pmatrix}0 & 10 \\ -10 & 0 \end{pmatrix}$,
$A_{1}=\begin{pmatrix}0 & 10 \\ -10 & 0 \end{pmatrix}$}
\includegraphics[width=80mm]
{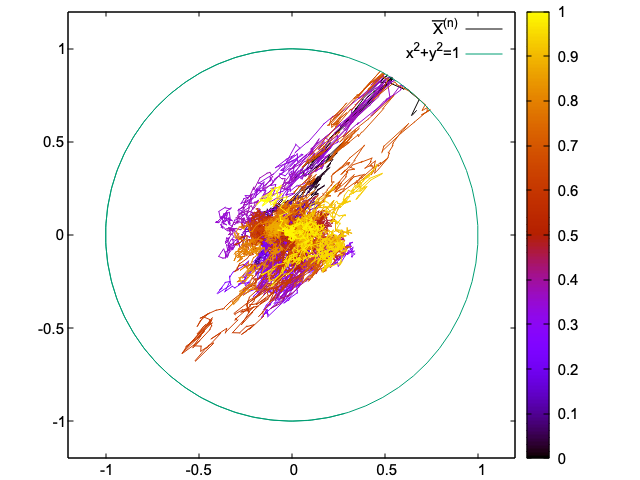}
\caption{\usebox{\smlmat}.}
\label{fig_8}
\end{minipage}
\end{figure}

Finally, we consider the difference between two solutions $Y^{1}$ and $Y^{2}$ of the system of SDE \eqref{eq:4} with different initial conditions, $\nu=\sqrt{2}$ and $A_{p}=0$, $p=0,1,\ldots,m$.
Swart \cite{Sw02} showed that the map $t \mapsto |Y^{1}(t)-Y^{2}(t)|$ is almost surely non-increasing if $\kappa \geq 1$ (see Theorem 3 in \cite{Sw02}).
Figure \ref{fig_9}, \ref{fig_10} are behaviors of the difference using the projection scheme \eqref{eq:8} for corresponding $X^{1}$ and $X^{2}$.
From these figures, we can confirm that the projection scheme expresses the theoretical result in \cite{Sw02}.

\begin{figure}[t!]
\begin{minipage}{0.5\hsize}
\includegraphics[width=80mm]
{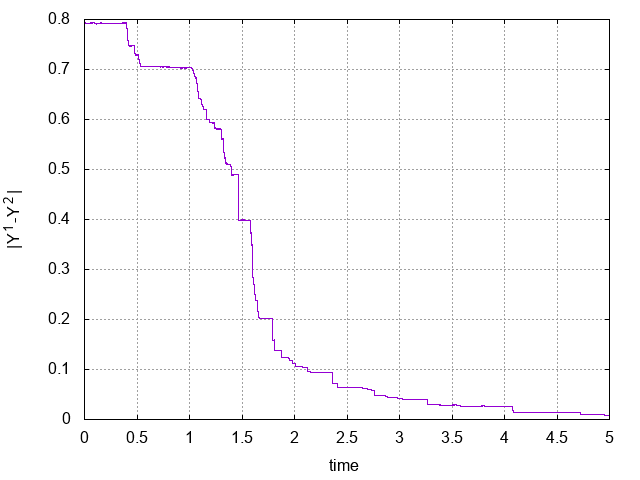}
\caption{$\kappa = 1$, the initial conditions $X^{1}(0)=(0,0)^{\top}$, $X^{2}(0)=(-0.7,0.2)^{\top}$, $n=312500$ time steps for $\overline{X}^{(n)}$.}
\label{fig_9}
\end{minipage}~
\begin{minipage}{0.5\hsize}
\includegraphics[width=80mm]
{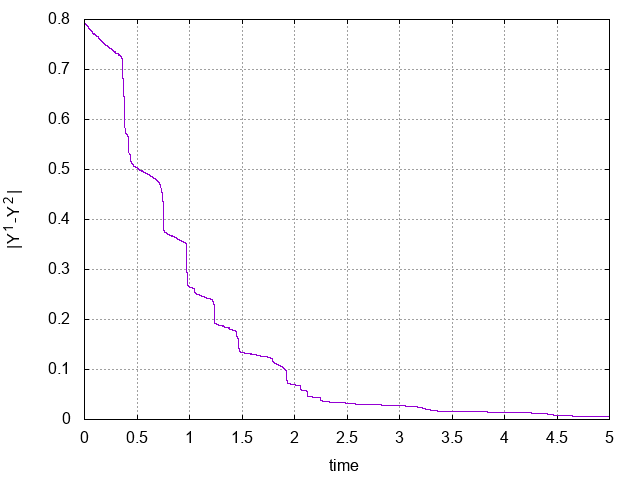}
\caption{$\kappa = 6/5$, the initial conditions $X^{1}(0)=(0,0)^{\top}$, $X^{2}(0)=(-0.7,0.2)^{\top}$, $n=112500$ time steps for $\overline{X}^{(n)}$.}
\label{fig_10}
\end{minipage}
\end{figure}

\section*{Acknowledgements}
The second author was supported by JSPS KAKENHI Grant Number 19K14552.
The third author was supported by JSPS KAKENHI Grant Numbers 17J05514 and 22K13965.

\end{document}